\documentclass[a4paper,11pt]{article}

\usepackage[english]{babel}
\usepackage[T1]{fontenc}
\usepackage{pifont}
\usepackage{ulem}

\usepackage{amsmath}
\usepackage{amssymb}
\usepackage{amsthm}
\usepackage{amscd}
\usepackage{mathrsfs}
\usepackage[all]{xy}
\usepackage{makeidx}
\usepackage[applemac]{inputenc}
\usepackage{enumerate}
\usepackage{relsize}
\usepackage{textcomp}

\usepackage{geometry}
\theoremstyle{definition}
\newtheorem{theorem}			     {Theorem}	    [section]
\newtheorem{proposition}  [theorem]	 {Proposition}	
	
\newtheorem{lemma}	      [theorem]  {Lemma}		
\newtheorem{definition}	  [theorem]  {Definition}

\DeclareMathOperator{\Img}{Im}

\DeclareMathOperator{\Lex}{Lex}
\DeclareMathOperator{\Set}{Set}
\DeclareMathOperator{\op}{op}
\DeclareMathOperator{\Alg}{Alg}
\DeclareMathOperator{\Ab}{Ab}
\DeclareMathOperator{\Mod}{Mod}

\DeclareMathOperator{\Cat}{Cat}

\DeclareMathOperator{\Fr}{Fr}
\DeclareMathOperator{\Def}{Def}

\DeclareMathOperator{\Mal}{Mal}
\DeclareMathOperator{\ev}{ev}
\DeclareMathOperator{\Sub}{Sub}
\DeclareMathOperator{\tw}{tw}
\DeclareMathOperator{\Pt}{Pt}

\newcommand{\CC}{\mathbb{C}}

\newcommand{\MM}{\mathbb{M}}

\newcommand{\T}{\mathcal{T}}

\newcommand{\PP}{\mathbb{P}}

\newcommand{\V}{\mathcal{V}}
\newcommand{\cd}{\xymatrix}

\newcommand{\pb}[1][dr]{\save*!/#1-1.5pc/#1:(-1,1)@^{|-}\restore}
\newcommand{\lpb}[1][ur]{\save*!/#1+1.5pc/#1:(1,1)@_{|-}\restore}

\title{An embedding theorem for regular Mal'tsev categories}
\author{Pierre-Alain Jacqmin\thanks{Financial support from FNRS grants 1.A741.14 and `Bref s\'ejour \`a l'\'etranger' and
partial support from South African National Research Foundation are gratefully acknowledged.}}

\date{15 February 2017}

\begin{document}

\maketitle

\vspace{-15pt}

\begin{abstract}

In this paper, we obtain a non-abelian analogue of Lubkin's embedding theorem for abelian categories. Our theorem faithfully embeds any small regular Mal'tsev
category $\mathbb{C}$ in an $n$-th power of a particular locally finitely presentable regular Mal'tsev category. The embedding preserves and reflects finite limits,
isomorphisms and regular epimorphisms, as in the case of Barr's embedding theorem for regular categories.
Furthermore, we show that we can take $n$ to be the (cardinal) number of subobjects of the terminal object in $\mathbb{C}$.

\end{abstract}

\small{\textit{2000 Mathematics Subject Classification.} 18B15, 08B05, 08A55, 18C35, 18C10.}

\section{Introduction}

One of the most famous results in category theory might be the Yoneda lemma and as a corollary, the Yoneda embedding. As we know, for every small category $\CC$,
it constructs a fully faithful embedding $\CC \hookrightarrow \Set^{\CC^{\op}}$ which preserves limits. Therefore, for some type of statements about limits,
one only needs to produce the proof in $\Set$ to get the result in any category.
One can thus say that the category $\Set$ of sets `represents' all categories.

Moreover, $\Set$ also represents regular categories (i.e., finitely complete categories with coequalisers of kernel pairs and pullback stable regular epimorphisms~\cite{barr2}).
Indeed, Barr's embedding theorem~\cite{barr} enables us to restrict to $\Set$ the proof of some statements about finite limits and regular epimorphisms in regular categories.
Note that the key ingredients of regularity are precisely about finite limits and regular epimorphisms.
In the same way, a wide range of statements about finite limits and finite colimits in abelian categories can be restricted to $\Ab$,
the category of abelian groups~\cite{lubkin,freyd,mitchell}.

Whereas abelian categories do not cover important algebraic examples such as the categories of groups or rings, the notion of a regular category is in some sense too general
because every (quasi)-variety of universal algebras is a regular category. Therefore, one needs some intermediate classes of categories to study the categorical properties of groups.
To achieve this, regular Mal'tsev categories have been introduced in~\cite{CLP} as regular categories in which the composition of equivalence relations is commutative.
This is equivalent to the property that every reflexive relation is an equivalence~\cite{CLP}. This condition is equivalent in the more general context of finitely complete categories
to the condition that each relation is difunctional; and this is the property that defines Mal'tsev categories in this context~\cite{CPP}. Their name comes from the mathematician
Mal'tsev who characterised~\cite{maltsev} (one-sorted finitary) algebraic categories in which this property holds as the ones whose corresponding theory admits a ternary term
$p(x,y,z)$ satisfying the identities $p(x,y,y)=x=p(y,y,x)$. These characterisations are recalled in Section~\ref{section regular mal'tsev categories}.

The aim of this paper is to prove an embedding theorem for regular Mal'tsev categories.
The `representing category' will be the category of models of a finitary essentially algebraic theory (i.e., a locally finitely presentable category~\cite{GU,AR}).
As in the algebraic case, objects in such categories are given by an $S$-sorted set $A$ (i.e., an object in $\Set^S$, for a set $S$ of sorts) endowed with finitary operations
$A_{s_1}\times \cdots \times A_{s_n} \rightarrow A_s$ satisfying some given equations. The difference is that some of these operations can be only partially defined
(and defined exactly for those $n$-tuples satisfying some given equations involving totally defined operations).
We recall this concept in Section~\ref{section finitary essentially algebraic categories}.
We also characterise there those categories of models which are regular and those that are Mal'tsev
(via a ternary Mal'tsev term as in the varietal case, see Theorem~\ref{mod gamma maltsev}).
We then construct the `representing' regular Mal'tsev category $\Mod(\Gamma_{\Mal})$ using those characterisations.

Section~\ref{section the embedding theorem} is devoted to the proof of our embedding Theorem~\ref{embedding theorem}:
every small regular Mal'tsev category $\CC$ admits a faithful embedding into $\Mod(\Gamma_{\Mal})^{\Sub(1)}$ which preserves and reflects finite limits,
isomorphisms and regular epimorphisms. Here, $\Sub(1)$ denotes the set of subobjects of the terminal object $1$ in $\CC$.
This proof uses three main ingredients: approximate Mal'tsev co-operations (introduced in~\cite{BJ} and recalled in Section
\ref{section regular mal'tsev categories}), the Yoneda embedding $\CC \hookrightarrow \Lex(\CC,\Set)^{\op}$ (which lifts the property of being regular Mal'tsev from
$\CC$ to $\Lex(\CC,\Set)^{\op}$, see Proposition~\ref{regular maltsev preserved}) and a $\CC$-projective covering of $\Lex(\CC,\Set)^{\op}$~\cite{barr,grothendieck 1}.
We can notice that, with this technique, we could also have embedded each small regular Mal'tsev category in a power of the category of approximate Mal'tsev algebras. These are pairs
of sets $A,B$ together with two operations $p \colon A^3 \rightarrow B$, $a \colon A \rightarrow B$ satisfying the axioms $p(x,y,y)=a(x)=p(y,y,x)$. However, this category is not
a Mal'tsev category and therefore we have had to refine this argument considering the essentially algebraic category $\Mod(\Gamma_{\Mal})$, which is a regular Mal'tsev category.

Due to this embedding theorem, one can reduce the proof of some propositions about finite limits and regular epimorphisms in a regular Mal'tsev category to 
the particular case of $\Mod(\Gamma_{\Mal})$. One is then allowed to use elements and (approximate) Mal'tsev operations to prove
some statements in a regular Mal'tsev context. An example of such an application is given in Section~\ref{section applications}.
Note that our embedding is not full, but it reflects isomorphisms, which is enough for such applications.
Indeed, fullness of an embedding $\CC \hookrightarrow \MM^\PP$ is not helpful when we look at the components $\ev_P \colon \MM^\PP \rightarrow \MM$
(which is what we do when we reduce a proof to $\MM$).

\subsection*{Acknowledgement}

The author would like to warmly thank Zurab Janelidze for proposing this project and helpful discussions about it.
The author is also grateful to Stellenbosch University for their kind hospitality during the (South-African) spring 2014.

\section{Regular Mal'tsev categories} \label{section regular mal'tsev categories}

A \textit{regular category} is a finitely complete category with coequalisers of kernel pairs and pullback stable regular epimorphisms~\cite{barr2}.
They admits a (regular epi, mono)-factorisation system (i.e., a factorisation system $(\mathcal{E},\mathcal{M})$ where $\mathcal{E}$ is the class of all regular epimorphisms and
$\mathcal{M}$ is the class of all monomorphisms). Moreover, in such categories, two relations $R \rightarrowtail X \times Y$ and $S \rightarrowtail Y \times Z$
can be composed to form a relation $S \circ R \rightarrowtail X \times Z$. This gives rise to the so called `calculus of relations'.

In this context of regular categories, Mal'tsev categories were introduced in~\cite{CLP}. Two years later, in~\cite{CPP}, the authors enlarged this notion of a Mal'tsev category
in the context of finitely complete categories.

\begin{definition} \cite{CPP}
A finitely complete category $\CC$ is a Mal'tsev category if every reflexive relation $r \colon R \rightarrowtail X \times X$ is an equivalence relation.
\end{definition}

In order to recall some well-known characterisations of Mal'tsev categories, we first need to recall what are difunctional relations.

\begin{definition} \cite{CPP}
A difunctional relation in a finitely complete category $\CC$ is an internal relation $r=(r_1,r_2) \colon R \rightarrowtail X \times Y$ such that,
when we consider the following diagram
where both squares are pullbacks and $\tw$ the twisting isomorphism,
$$\cd{S \pb \ar[rr] \ar@{}[d]|(.3){}="A" \ar@{>->} "A";[d] && R \ar@{}[d]|(.3){}="B" \ar@{>->} "B";[d]^-{r} & T \lpb \ar[l] \ar@{}[d]|(.3){}="C" \ar@{>->} "C";[d]\\
R \times R \ar[r]_-{r_2 \times r_1} & Y \times X \ar[r]_-{\tw} & X \times Y & R \times R \ar[l]^-{r_1 \times r_2}}$$
the canonical monomorphism $S \cap T \rightarrowtail T$ is an isomorphism.
\end{definition}

In the category $\Set$ of sets (or any algebraic category), a relation $R \subseteq X \times Y$ is difunctional if it satisfies
$$(xRy \wedge xRy' \wedge x'Ry') \Rightarrow x'Ry$$ for all $x,x' \in X$ and $y,y' \in Y$.

\begin{theorem}\cite{CPP}
Let $\CC$ be a finitely complete category. The following statements are equivalent.
\begin{enumerate}
\item $\CC$ is a Mal'tsev category.
\item Any reflexive relation in $\CC$ is symmetric.
\item Any reflexive relation in $\CC$ is transitive.
\item Any relation $r \colon R \rightarrowtail X \times Y$ is difunctional.
\item Any relation $r \colon R \rightarrowtail X \times X$ is difunctional.
\end{enumerate}
\end{theorem}

In a regular context, we have even more characterisations.

\begin{theorem}\cite{CLP} \label{characterisations mal'tsev for regular}
Let $\CC$ be a regular category. The following statements are equivalent.
\begin{enumerate}
\item $\CC$ is a Mal'tsev category.
\item The composite of two equivalence relations on the same object is an equivalence relation.
\item If $R$ and $S$ are equivalence relations on the same object, then $R \circ S = S \circ R$.
\item \label{mal'tsev reflexive graphs} For every reflexive graph,
$$\cd{G \ar@<2pt>[r]^-{d} \ar@<-2pt>[r]_-{c} & X \ar@<3pt>@/^0.9pc/[l]^-{s}}$$
the pullback of $(d,c)$ along $(c,d)$ is a regular epimorphism.
$$\cd{P \pb \ar[r] \ar@{->>}[d] & G \ar[d]^-{(d,c)} \\ G \ar[r]_-{(c,d)} & X \times X}$$
\end{enumerate}
\end{theorem}

The name of Mal'tsev categories comes from the following result of A.I.\ Mal'tsev.

\begin{theorem}\cite{maltsev} \label{maltsev theorem}
Let $\T$ be a finitary one-sorted algebraic theory. Then, the category of $\T$-algebras $\Alg_{\T}$ is a Mal'tsev category if and only if $\T$ contains a ternary term $p(x,y,z)$
satisfying the identities $$p(x,y,y)=x=p(y,y,x).$$
\end{theorem}

Finally, we will need one more characterisation of Mal'tsev categories in a regular context.
In~\cite{BJ}, the authors define an \textit{approximate Mal'tsev co-operation on $X$}
(for an object $X$ in a finitely complete category with binary coproducts) as a morphism
$p \colon Y \rightarrow 3X$ together with an \textit{approximation $a \colon Y \rightarrow X$} such that the square
$$\cd{Y \ar[r]^-{p} \ar[d]_-{a} & 3X \ar[d]^-{\left( \begin{smallmatrix} \iota_1 & \iota_1 \\ \iota_2 & \iota_1 \\ \iota_2 & \iota_2 \end{smallmatrix}\right)} \\
X \ar[r]_-{(\iota_1,\iota_2)} & (2X)^2}$$
commutes. For each object $X$, one can build the \textit{universal approximate Mal'tsev co-operation $(p^X,a^X)$ on $X$} by considering the following pullback.
$$\cd{M(X) \pb \ar[r]^-{p^X} \ar[d]_-{a^X} & 3X \ar[d]^-{\left( \begin{smallmatrix} \iota_1 & \iota_1 \\ \iota_2 & \iota_1 \\ \iota_2 & \iota_2 \end{smallmatrix}\right)} \\
X \ar[r]_-{(\iota_1,\iota_2)} & (2X)^2}$$

\begin{theorem} [Theorem~4.2 in~\cite{BJ}] \label{approximate co-operations}
Let $\CC$ be a regular category with binary coproducts. The following statements are equivalent:
\begin{enumerate}
\item \label{approx 1} $\CC$ is a Mal'tsev category.
\item \label{approx 2} For each $X \in \CC$, there is an approximate Mal'tsev co-operation on $X$ for which the approximation $a$ is a regular epimorphism.
\item \label{approx 3} For each $X \in \CC$, the universal approximate Mal'tsev co-operation on $X$ is such that the approximation $a^X$ is a regular epimorphism.
\end{enumerate}
\end{theorem}

\section{Finitary essentially algebraic categories} \label{section finitary essentially algebraic categories}

A \textit{locally finitely presentable category} is a cocomplete category which has a strong set of generators formed by finitely presentable objects.
We know from~\cite{GU} that locally finitely presentable categories are, up to equivalence, exactly the categories of the form $\Lex(\CC,\Set)$
for a small finitely complete category $\CC$ (i.e., the category of finite limit preserving functors from $\CC$ to $\Set$).
Moreover, they have been further characterised as `finitary essentially algebraic categories'.

\subsection{Essentially algebraic theories and their models}

An essentially algebraic category is, roughly speaking, a category of (many-sorted) algebraic structures with partial operations.
The domains of definition of these partial operations are themselves defined as the solution sets of some totally defined equations.

More precisely, an \textit{essentially algebraic theory} is a quintuple $\Gamma =(S,\Sigma,E,\Sigma_t,\Def)$ where $S$ is a set (the set of sorts),
$\Sigma$ is an $S$-sorted signature of algebras, $E$ is a set of $\Sigma$-equations, $\Sigma_t \subseteq \Sigma$ is the subset of `total operation symbols' and $\Def$
is a function assigning to each operation symbol $\sigma \colon \prod_{i \in I} s_i \rightarrow s$ in $\Sigma \setminus \Sigma_t$ a set $\Def(\sigma)$ of $\Sigma_t$-equations in the
variables $x_i$ of sort $s_i$ ($i \in I$).

A \textit{model} $A$ of an essentially algebraic theory $\Gamma$ is an $S$-sorted set $(A_s)_{s \in S} \in \Set^S$ together with, for each operation symbol
$\sigma \colon \prod_{i \in I} s_i \rightarrow s$ in $\Sigma$, a partial function $$\sigma^A \colon \prod_{i \in I} A_{s_i} \rightarrow A_s$$ such that:
\begin{itemize}
\item for each $\sigma \in \Sigma_t$, $\sigma^A$ is defined everywhere,
\item for each $\sigma \colon \prod_{i \in I} s_i \rightarrow s$ in $\Sigma \setminus \Sigma_t$ and any $(a_i \in A_{s_i})_{i \in I}$, $\sigma^A((a_i)_{i \in I})$
is defined if and only if the elements $a_i$'s satisfy the equations of $\Def(\sigma)$ in $A$,
\item $A$ satisfies the equations of $E$ wherever they are defined.
\end{itemize}

If $A$ and $B$ are two $\Gamma$-models, a \textit{homomorphism $f \colon A \rightarrow B$ of models} is an $S$-sorted function $(f_s \colon A_s \rightarrow B_s)_{s \in S}$ such that,
for each $\sigma \colon \prod_{i \in I} s_i \rightarrow s$ in $\Sigma$ and any $(a_i \in A_{s_i})_{i \in I}$ such that $\sigma^A((a_i)_{i \in I})$ is defined,
\begin{equation} \label{homomorphism of models}
f_s(\sigma^A((a_i)_{i \in I}))=\sigma^B((f_{s_i}(a_i))_{i \in I}).
\end{equation}
Notice that if (\ref{homomorphism of models}) holds for all $\sigma \in \Sigma_t$, then, for each $\sigma' \in \Sigma \setminus \Sigma_t$, $\sigma'^B((f_{s_i}(a_i))_{i \in I})$
is defined if $\sigma'^A((a_i)_{i \in I})$ is, while the converse does not hold in general. The category of $\Gamma$-models and their homomorphisms is denoted by $\Mod(\Gamma)$.
A category which is equivalent to some model category $\Mod(\Gamma)$ for an essentially algebraic theory $\Gamma$ is called \textit{essentially algebraic}.

If all arities of $\Sigma$ are finite, if each equation of $E$ and of all $\Def(\sigma)$'s uses only a finite number of variables and if all sets $\Def(\sigma)$'s are also finite,
$\Gamma$ is called a \textit{finitary essentially algebraic theory}.
A category which is equivalent to some category $\Mod(\Gamma)$ for a finitary essentially algebraic theory $\Gamma$ is called a \textit{finitary essentially algebraic category}.
As mentioned above, they are exactly the locally finitely presentable categories.

\begin{theorem} \cite{GU,AR} \label{characterisation of locally finitely presentable categories}
A category is locally finitely presentable if and only if it is a finitary essentially algebraic category.
\end{theorem}

The basic examples of finitary essentially algebraic categories are finitary (many-sorted) quasi-varieties and so in particular finitary (many-sorted) varieties.
The category $\Cat$ of small categories is also finitary essentially algebraic.

Let us now focus our attention to the forgetful functor $U \colon \Mod(\Gamma) \rightarrow \Set^S$ for a finitary essentially algebraic theory $\Gamma$.
As expected, we can easily prove that $U$ creates small limits so that they exists in $\Mod(\Gamma)$ and are computed in each sort as in $\Set$.
Moreover, $U$ preserves and reflects monomorphisms and isomorphisms. Thus a homomorphism of $\Gamma$-models is a monomorphism (resp.~an isomorphism) if and only if
it is injective (resp.~bijective) in each sort. In view of Theorem~\ref{characterisation of locally finitely presentable categories}, $\Mod(\Gamma)$ is also cocomplete,
but colimits are generally harder to describe. In addition, we do not have an easy characterisation of regular epimorphisms and $\Mod(\Gamma)$ is in general not regular.

We are now going to describe a left adjoint for the forgetful functor $U \colon \Mod(\Gamma) \rightarrow \Set^S$.
In order to do so, we refer the reader to~\cite{AR} for the definition of terms in $\Sigma$ in the variables of an $S$-sorted set $X$.
By abuse of notation, by a (finitary) term $\tau \colon \prod_{i=1}^n s_i \rightarrow s$ of $\Sigma$,
we mean a term of $\Sigma$ of sort $s$ over the $S$-sorted set $X$ which contains exactly one formal symbol $x_i$ of sort $s_i$ for each $1 \leqslant i \leqslant n$.
Since $\Mod(\Gamma)$ is cocomplete, $U$ has a left adjoint as long as a reflection of those finite $S$-sorted sets $X$ along $U$ exists. Let us describe it.
If $\tau,\tau' \colon \prod_{i=1}^n s_i \rightarrow s$ are two terms of $\Sigma$, we say that $\tau=\tau'$ is a \textit{theorem} of $\Gamma$ if
$\tau(a_1, \dots, a_n)=\tau'(a_1, \dots, a_n)$ holds in any $\Gamma$-model $A$ and for any interpretation of the variables of $X$ in $A$
(i.e., $S$-sorted function $X \rightarrow U(A)$) such that both $\tau(a_1, \dots, a_n)$ and $\tau'(a_1, \dots, a_n)$ are defined.
Then, we define the set of \textit{everywhere-defined terms $\prod_{i=1}^n s_i \rightarrow s$}
as the smallest subset of the set of terms of $\Sigma$ in the variables of $X$ such that:
\begin{itemize}
\item for each element $x_k \in X_{s_k}$, the $k$-th projection $\prod_{i=1}^n s_i \rightarrow s_k$ is an everywhere-defined term,
\item if $(\tau_j \colon \prod_{i=1}^{n} s_i \rightarrow s^j)_{j \in \{1, \dots, m\}}$ are
everywhere-defined terms and $\sigma \colon \prod_{j=1}^m s^j \rightarrow s$ is an operation symbol
of $\Sigma$ such that, either $\sigma \in \Sigma_t$ or $\sigma \in \Sigma \setminus \Sigma_t$ and, for each equation $(\mu,\mu')$ of $\Def(\sigma)$,
$\mu(\tau_1, \dots, \tau_n)=\mu'(\tau_1, \dots, \tau_n)$ is a theorem of $\Gamma$,
then $\sigma(\tau_1, \dots, \tau_n) \colon \prod_{i=1}^n s_i \rightarrow s$ is everywhere-defined.
\end{itemize}
Now, $\Fr(X)_s$ is the set of equivalence classes of everywhere-defined terms $\tau \colon \prod_{i=1}^n s_i \rightarrow s$ of $\Sigma$,
where we identify the terms $\tau$ and $\tau'$ if and only if $\tau=\tau'$ is a theorem of $\Gamma$. The operations on $\Fr(X)$ and
the $S$-sorted function $X \rightarrow U(\Fr(X))$ are defined in the obvious way. The fact that this $S$-sorted function is the reflection of $X$ along $U$
can be deduced easily from the definitions. We thus have an adjunction $\Fr \dashv U$.

\subsection{Regular Mal'tsev finitary essentially algebraic categories}

For a small finitely complete category $\CC$, $\Lex(\CC,\Set)$ is a regular category if and only if $\CC$ is weakly coregular~\cite{CPR}.
Moreover, in~\cite{GP}, the authors describe the categories $\CC$ for which $\Lex(\CC,\Set)$ is a regular Mal'tsev category.
On the other hand, the categories of the form $\Mod(\Gamma)$
for a finitary $\Gamma$ can equivalently be written as $\Lex(\CC,\Set)$ for some small finitely complete category $\CC$.
This category $\CC$ can be chosen (up to equivalence) as the dual of the full subcategory of finitely presentable objects in $\Mod(\Gamma)$~\cite{GU}.
However, those objects are hard to describe in general and it is not easy to derive a direct characterisation of those $\Gamma$'s
for which $\Mod(\Gamma)$ is regular or regular Mal'tsev from the previous ones.
In this subsection we give a direct characterisation of those finitary essentially algebraic theories whose categories of models are regular, and separately,
those whose categories of models are Mal'tsev categories. The more general case of a (non necessarily finitary) essentially algebraic theory is similar and appears in~\cite{jacqmin}.

We start by describing a (strong epi, mono)-factorisation system in $\Mod(\Gamma)$, for an arbitrary finitary essentially algebraic theory $\Gamma$
(by a (strong epi, mono)-factorisation system we mean a factorisation system $(\mathcal{E},\mathcal{M})$ where $\mathcal{E}$ is the class of all strong epimorphisms and
$\mathcal{M}$ is the class of all monomorphisms). Let $f \colon A \rightarrow B$ be a homomorphism of $\Gamma$-models.
By abuse of notation, we will often write $f(a)$ instead of $f_s(a)$ for $s \in S$ and $a \in A_s$.
We consider the submodel $I$ of $B$ such that, for all $s \in S$, \label{image}
\begin{align*}
I_s=\lbrace \tau(f(a_1), \dots, f(a_n)) \mid
&a_i \in A_{s_i} \text{ and } \tau \colon \prod_{i=1}^n s_i \rightarrow s \text{ is a finitary term of } \Sigma \text{ which}\\
& \text{is defined in } B \text{ on } (f(a_1), \dots, f(a_n)) \rbrace.
\end{align*}
We can thus view $I$ as the smallest submodel of $B$ for which $f_s(a) \in I_s$ for all $s \in S$ and $a \in A_s$.
This means that the corestriction $p \colon A \rightarrow I$ of $f$ to $I$ is a strong epimorphism and $f$ factors as $f=ip$ with $i$ the inclusion $I \hookrightarrow B$.
As usual, we will refer to $I=\Img(f)$ as the \textit{image of $f$}.

Before being able to describe those $\Gamma$'s for which $\Mod(\Gamma)$ is regular, we need the following lemma.

\begin{lemma} \label{lemma for regularity}
Let $\Gamma$ be a finitary essentially algebraic theory and $\theta \colon \prod_{i=1}^n s_i \rightarrow s$ a finitary term of $\Gamma$.
If $(a_i \in A_{s_i})_{i \in \{1,\dots,n\}}$ are elements of a
$\Gamma$-model $A$, we can find a strong epimorphism $q \colon A \twoheadrightarrow B$ in $\Mod(\Gamma)$ such that $\theta(q(a_1),\dots,q(a_n))$ is defined and
if $f \colon A \rightarrow C$ is a homomorphism such that $\theta(f(a_1),\dots,f(a_n))$ is defined, then $f$ factors uniquely through $q$.
\end{lemma}

\begin{proof}
We are going to prove this lemma by induction on the number of steps used in the construction of the term $\theta$.
If $\theta$ is a projection (or any everywhere-defined term), $1_A$ is the homomorphism we are looking for.
Now, suppose $\theta$ uses the operation symbols or projections
$\sigma_1, \dots, \sigma_m \in \Sigma \cup \{p_k \colon \prod_{i=1}^n s_i \rightarrow s_k \, | \, 1 \leqslant k \leqslant n\}$ as first step of its construction.
Thus, $\theta$ can be written as
$$\theta(x_1,\dots,x_n)=\theta'(\sigma_1(x_1,\dots,x_n),\dots,\sigma_m(x_1,\dots,x_n))$$
where $\theta'$ uses less steps than $\theta$ to be constructed. Let $R$ be the smallest submodel of $A \times A$
which contains $(\chi(a_1,\dots,a_n),\chi'(a_1,\dots,a_n))$ for all equations $(\chi,\chi') \in \Def(\sigma_j)$ and all $j$
such that $\sigma_j \in \Sigma \setminus \Sigma_t$. Let $q_1$ be the coequaliser of $r_1$ and $r_2$
$$\cd{R \ar@<2pt>[r]^-{r_1} \ar@<-2pt>[r]_-{r_2} & A \ar@{->>}[r]^-{q_1} & B_1}$$
where $r_i=p_ir$ with $r$ the inclusion $R \hookrightarrow A \times A$ and $p_1$ and $p_2$ the projections.
Thus, in $B_1$, all $\sigma_j(q_1(a_1), \dots, q_1(a_n))$ are defined. Now, we use the induction hypothesis on $\theta'$ to build a universal strong epimorphism
$q_2 \colon B_1 \twoheadrightarrow B$ such that
\begin{align*}
&\theta'(q_2(\sigma_1(q_1(a_1),\dots,q_1(a_n))),\dots,q_2(\sigma_m(q_1(a_1),\dots,q_1(a_n))))\\
&= \theta'(\sigma_1(q_2q_1(a_1),\dots,q_2q_1(a_n)),\dots,\sigma_m(q_2q_1(a_1),\dots,q_2q_1(a_n)))\\
&= \theta(q_2q_1(a_1),\dots,q_2q_1(a_n))
\end{align*}
is defined. Let us prove that $q_2q_1$ is the strong epimorphism we are looking for. Let $f \colon A \rightarrow C$ be a homomorphism such that $\theta(f(a_1),\dots,f(a_n))$
is defined. Since the kernel pair $R[f]$ of $f$ contains $(\chi(a_1,\dots,a_n),\chi'(a_1,\dots,a_n))$ for all equations $(\chi,\chi')\in \Def(\sigma_j)$ and all $j$
such that $\sigma_j \in \Sigma \setminus \Sigma_t$, we have $R \subseteq R[f]$ and $fr_1=fr_2$. Therefore, $f$ factors through $q_1$ as $f=gq_1$.
Finally, $g$ factors through $q_2$ since 
\begin{align*}
&\theta(f(a_1), \dots, f(a_n))\\
&= \theta(gq_1(a_1), \dots, gq_1(a_n))\\
&= \theta'(\sigma_1(gq_1(a_1),\dots,gq_1(a_n)),\dots,\sigma_m(gq_1(a_1),\dots,gq_1(a_n)))\\
&= \theta'(g(\sigma_1(q_1(a_1),\dots,q_1(a_n))),\dots,g(\sigma_m(q_1(a_1),\dots,q_1(a_n))))
\end{align*}
is defined.
\end{proof}

We are now able to describe regular finitary essentially algebraic categories.

\begin{proposition} \label{mod gamma regular}
Let $\Gamma$ be a finitary essentially algebraic theory. Then $\Mod(\Gamma)$ is a regular category if and only if, for each finitary
term $\theta \colon \prod_{i=1}^n s_i \rightarrow s$ of $\Gamma$, there exists in~$\Gamma$:
\begin{itemize}
\item a finitary term $\pi \colon \prod_{j=1}^m s'_j \rightarrow s$,
\item for each $1 \leqslant j \leqslant m$, an everywhere-defined term $\alpha_j \colon s \rightarrow s'_j$ and
\item for each $1 \leqslant j \leqslant m$, an everywhere-defined term $\mu_j \colon \prod_{i=1}^n s_i \rightarrow s'_j$ 
\end{itemize}
such that
\begin{itemize}
\item $\pi(\alpha_1(x),\dots,\alpha_m(x))$ is everywhere-defined,
\item $\pi(\alpha_1(x),\dots,\alpha_m(x))=x$ is a theorem of $\Gamma$,
\item $\alpha_j(\theta(x_1,\dots,x_n))=\mu_j(x_1,\dots,x_n)$ is a theorem of $\Gamma$ for each $1 \leqslant j \leqslant m$.
\end{itemize}
\end{proposition}

\begin{proof}
Since $\Mod(\Gamma)$ is complete and has a (strong epi, mono)-factorisation system, it is regular if and only if strong epimorphisms are pullback stable
(see e.g.~Proposition~2.2.2 in the second volume of~\cite{borceux}).
So, let us suppose that the condition in the statement holds in~$\Gamma$ and consider a pullback square
$$\cd{P \pb \ar[r]^{p'} \ar[d]_-{f'} & B \ar[d]^-{f} \\ A \ar@{->>}[r]_-{p} & C}$$
in $\Mod(\Gamma)$ with $p$ a strong epimorphism. We have to prove that $\Img(p')=B$. So, let $b \in B_s$ for some $s \in S$. Since $p$ is a strong epimorphism,
there exists a finitary term $\theta \colon \prod_{i=1}^n s_i \rightarrow s$ of $\Sigma$ and elements $a_i \in A_{s_i}$ for each $1 \leqslant i \leqslant n$ such that
$\theta(p(a_1), \dots, p(a_n))$ is defined and is equal to $f(b)$. Let the terms $\pi$, $\alpha_j$'s and $\mu_j$'s be given by the assumption for this $\theta$.
For each $j \in \{1,\dots,m\}$,
\begin{align*}
f(\alpha_j(b))&=\alpha_j(f(b))=\alpha_j(\theta(p(a_1), \dots, p(a_n)))
\\&=\mu_j(p(a_1), \dots, p(a_n))=p(\mu_j(a_1,\dots,a_n))
\end{align*}
since $\alpha_j$ and $\mu_j$ are everywhere-defined. But small limits in $\Mod(\Gamma)$ are computed in each sort as in $\Set$.
Hence, $d_j=(\mu_j(a_1,\dots,a_n),\alpha_j(b)) \in P_{s'_j}$ with $$b=\pi(\alpha_1(b),\dots,\alpha_m(b))=\pi(p'(d_1),\dots,p'(d_m)).$$
Therefore, $b \in \Img(p')_s$ and $p'$ is a strong epimorphism.

Conversely, let us suppose that $\Mod(\Gamma)$ is regular and let $\theta \colon \prod_{i=1}^n s_i \rightarrow s$ be a finitary term of $\Sigma$.
Let $X$ be the $S$-sorted set which contains exactly, for each $i \in \{1,\dots,n\}$, an element $x_i$ of sort $s_i$ and $Y$ the $S$-sorted set which
contains exactly one element $y$ of sort $s$. We consider also the strong epimorphism $q \colon \Fr(X) \twoheadrightarrow B$ given by Lemma~\ref{lemma for regularity},
for the term $\theta$ and the elements $x_i \in \Fr(X)_{s_i}$. Thus $\theta(q(x_1),\dots,q(x_n))$ is defined.
Let $f \colon \Fr(Y) \rightarrow B$ be the unique map such that $f(y)=\theta(q(x_1),\dots,q(x_n))$ and consider the pullback of $q$ along $f$.
$$\cd{P \ar@{->>}[r]^-{p} \ar[d] \pb & \Fr(Y) \ar[d]^-{f} \\ \Fr(X) \ar@{->>}[r]_-{q} & B}$$
Since $\Mod(\Gamma)$ is regular, $p$ is also a strong epimorphism. So, $y \in \Img(p)_s$ which means, using the descriptions of $P$, $\Fr(X)$ and $\Fr(Y)$,
that there exist finitary terms $\pi \colon \prod_{j=1}^m s'_j \rightarrow s$, $\alpha_j \colon s \rightarrow s'_j$ and $\mu_j \colon \prod_{i=1}^n s_i \rightarrow s'_j$
for each $1 \leqslant j \leqslant m$ such that the $\alpha_j$'s and $\mu_j$'s are everywhere-defined, the equalities 
\begin{equation} \label{equation in y}
\begin{split}
y &= \pi(p(\mu_1(x_1,\dots,x_n),\alpha_1(y)),\dots,p(\mu_m(x_1,\dots,x_n),\alpha_m(y)))\\
&=\pi(\alpha_1(y),\dots,\alpha_m(y))
\end{split}
\end{equation}
hold and, for each $j \in \{1,\dots,m\}$,
\begin{equation} \label{equations in x}
\begin{split}
\mu_j(q(x_1),\dots,q(x_n))&=q(\mu_j(x_1,\dots,x_n))\\
&=f(\alpha_j(y))\\
&=\alpha_j(f(y))\\
&=\alpha_j(\theta(q(x_1),\dots,q(x_n))).
\end{split}
\end{equation}
Equalities (\ref{equation in y}) mean that $\pi(\alpha_1(x),\dots,\alpha_m(x))$ is everywhere-defined and $$\pi(\alpha_1(x),\dots,\alpha_m(x))=x$$ is a theorem of $\Gamma$.
With the universal properties of $\Fr(X)$ and $q$, equalities (\ref{equations in x}) mean that
$\alpha_j(\theta(a_1,\dots,a_n))=\mu_j(a_1,\dots,a_n)$ holds in any $\Gamma$-model as soon as $\theta(a_1,\dots,a_n)$ is defined.
\end{proof}

We characterise now those $\Gamma$'s for which $\Mod(\Gamma)$ is a Mal'tsev category. This theorem can be seen as a generalisation of Mal'tsev's Theorem~\ref{maltsev theorem}.

\begin{theorem} \label{mod gamma maltsev}
Let $\Gamma$ be a finitary essentially algebraic theory. Then $\Mod(\Gamma)$ is a Mal'tsev category if and only if, for each sort $s \in S$, there exists in $\Gamma$
a term $p^s \colon s^3 \rightarrow s$ such that
\begin{itemize}
\item $p^s(x,x,y)$ and $p^s(x,y,y)$ are everywhere-defined and
\item $p^s(x,x,y)=y$ and $p^s(x,y,y)=x$ are theorems of $\Gamma$.
\end{itemize}
\end{theorem}

\begin{proof}
Since finite limits in $\Mod(\Gamma)$ are computed in each sort as in $\Set$, a relation $R \rightarrowtail A \times B$ in $\Mod(\Gamma)$
can be seen as a submodel of $A \times B$
and it is difunctional exactly when, for each sort $s \in S$, $R_s \subseteq A_s \times B_s$ is difunctional as a relation in $\Set$.

Suppose first that such terms are given. Let $R \subseteq A \times B$ be a binary relation in $\Mod(\Gamma)$, $s \in S$, $a,a' \in A_s$ and $b,b' \in B_s$ such that
$(a,b),(a,b')$ and $(a',b')$ are in $R_s$. Since $p^s(a,a,a') \in A_s$ and $p^s(b,b',b') \in B_s$ are defined, so is $p^s((a,b),(a,b'),(a',b'))$ in the product $A \times B$.
Thus, $$(a',b)=(p^s(a,a,a'),p^s(b,b',b'))=p^s((a,b),(a,b'),(a',b')) \in R_s$$ and $R$ is difunctional.

Conversely, let us suppose that $\Mod(\Gamma)$ is a Mal'tsev category. Let $s \in S$ be a sort and $X$ the $S$-sorted set such that $X_s=\{x,y\}$ and
$X_{s'}=\varnothing$ for $s' \neq s$. We denote by $R$ the smallest homomorphic binary relation on $\Fr(X)$ which contains $(x,x),(x,y)$ and $(y,y)$.
It is easy to prove that this submodel of $\Fr(X) \times \Fr(X)$ is actually given by
\begin{align*}
R_{s'}=\{(\tau(x,x,y),\tau(x,y,y)) \, | \, &\tau \colon s^3 \rightarrow s' \text{ is a term and}\\
&\tau(x,x,y) \text{ and } \tau(x,y,y) \text{ are everywhere-defined terms}\}
\end{align*}
for all $s' \in S$. Since $\Mod(\Gamma)$ is supposed to be a Mal'tsev category, $R$ is difunctional and $(y,x) \in R_s$. This gives the expected term $p^s$.
\end{proof}

\subsection{A finitary essentially algebraic regular Mal'tsev category} \label{section a finitary essentially algebraic regular Mal'tsev category}

In this subsection we are going to construct a finitary essentially algebraic theory $\Gamma_{\Mal}$ for which $\Mod(\Gamma_{\Mal})$ is a regular Mal'tsev category.
This category of models is the one we need for our embedding theorem.

Firstly, if $\Gamma$ and $\Gamma'$ are two finitary essentially algebraic theories, we will write $\Gamma \subseteq \Gamma'$ to mean $S \subseteq S'$,
$\Sigma \subseteq \Sigma'$, $E \subseteq E'$, $\Sigma_t \subseteq \Sigma'_t$, $\Sigma \setminus \Sigma_t \subseteq \Sigma' \setminus \Sigma'_t$
and $\Def(\sigma)=\Def'(\sigma)$ for all $\sigma \in \Sigma \setminus \Sigma_t$. In this case, we have a forgetful functor $U \colon \Mod(\Gamma') \rightarrow \Mod(\Gamma)$.

We are going to construct recursively a series of finitary essentially algebraic theories
$$\Gamma^0 \subseteq \Delta^1 \subseteq \cdots \subseteq \Gamma^n \subseteq \Delta^{n+1} \subseteq \cdots$$
We define $\Gamma^0$ as $S^0=\{\star\}$ and $\Sigma^0=\Sigma^0_t=E^0=\varnothing$. Thus $\Mod(\Gamma^0) \cong \Set$.
Now, suppose we have defined
$$\Gamma^0 \subseteq \Delta^1 \subseteq \cdots \subseteq \Delta^n \subseteq \Gamma^n$$
with $\Gamma^n=(S^n,\Sigma^n,E^n, \Sigma_t^n,\Def^n)$. We are going to construct
$$\Delta^{n+1} = (S^{'n+1},\Sigma^{'n+1},E^{'n+1}, \Sigma_t^{'n+1},\Def^{'n+1})$$ first (below, $S^{-1}=\varnothing$):
\begin{itemize}
\item[] $$S^{'n+1}=S^n \cup \{(s,0),(s,1) \, | \, s \in S^n \setminus S^{n-1}\} \cong S^n \sqcup (S^n \setminus S^{n-1}) \sqcup (S^n \setminus S^{n-1}),$$
\item[]
\begin{align*}
\Sigma_t^{'n+1} = \Sigma^n_t &\cup \{\alpha^{s} \colon s \rightarrow (s,0) \, | \, s \in S^n \setminus S^{n-1}\}\\
&\cup \{\rho^s \colon s^3 \rightarrow (s,0) \, | \, s \in S^n \setminus S^{n-1}\}\\
&\cup \{\eta^{s},\varepsilon^{s} \colon (s,0) \rightarrow (s,1) \, | \, s \in S^n \setminus S^{n-1}\},
\end{align*}
\item[] $$\Sigma^{'n+1} = \Sigma^n \cup \Sigma_t^{'n+1} \cup \{\pi^{s} \colon (s,0) \rightarrow s \, | \, s \in S^n \setminus S^{n-1}\},$$
\item[]
\begin{align*}
E^{'n+1}= E^n &\cup \{\rho^s(x,y,y)=\alpha^s(x) \, | \, s \in S^n \setminus S^{n-1}\}\\
&\cup \{\rho^s(x,x,y)=\alpha^s(y) \, | \, s \in S^n \setminus S^{n-1}\}\\
&\cup \{\eta^{s}(\alpha^{s}(x))=\varepsilon^{s}(\alpha^{s}(x)) \, | \, s \in S^n \setminus S^{n-1}\}\\
&\cup \{\pi^{s}(\alpha^{s}(x))=x \, | \, s \in S^n \setminus S^{n-1}\}\\
&\cup \{\alpha^{s}(\pi^{s}(x))=x \, | \, s \in S^n \setminus S^{n-1}\},
\end{align*}
\item[]
$$\left\{\begin{array}{ll} \Def^{'n+1}(\sigma) = \Def^n(\sigma) \text{ if } \sigma \in \Sigma^n \setminus \Sigma^n_t\\
\Def^{'n+1}(\pi^{s}) = \{\eta^{s}(x)=\varepsilon^{s}(x)\} \text{ for } s \in S^n \setminus S^{n-1}. \end{array}\right.$$
\end{itemize}
This gives $\Gamma^n \subseteq \Delta^{n+1}$.
$$\cd{s^3 \ar@<2pt>[r]^-{\rho^s} & (s,0) \ar@<2pt>@{}[ld]_(.15){}="A" \ar@<1pt>@{.>}@/^/ "A";[ld]^(.4){\pi^s} \ar@<2pt>[r]^-{\eta^s} \ar@<-2pt>[r]_-{\varepsilon^s} & (s,1)\\
s \ar[ru] \ar@{}@<5pt>[ru]|-{\alpha^s}&&}$$

Now let $T^{n+1}$ be the set of finitary terms $\theta \colon \prod_{i=1}^m s_i \rightarrow s$ of $\Sigma^{'n+1}$ which are not terms of $\Sigma^{'n}$
(where we consider $\Sigma^{'0}=\varnothing$). We then define $\Gamma^{n+1}$ as:
\begin{itemize}
\item[] $$S^{n+1}= S^{'n+1} \cup \{s_{\theta},s'_{\theta} \, | \, \theta \in T^{n+1}\} \cong S^{'n+1} \sqcup T^{n+1} \sqcup T^{n+1},$$
\item[]
\begin{align*}
\Sigma^{n+1}_t = \Sigma^{'n+1}_t &\cup \{\alpha_{\theta} \colon s \rightarrow s_{\theta} \, | \, \theta \colon \prod_{i=1}^m s_i \rightarrow s \in T^{n+1}\}\\
&\cup \{\mu_{\theta} \colon \prod_{i=1}^m s_i \rightarrow s_{\theta} \, | \, \theta \colon \prod_{i=1}^m s_i \rightarrow s \in T^{n+1}\}\\
&\cup \{\eta_{\theta},\varepsilon_{\theta} \colon s_{\theta} \rightarrow s'_{\theta} \, | \,  \theta \in T^{n+1}\},
\end{align*}
\item[] $$\Sigma^{n+1} = \Sigma^{'n+1} \cup \Sigma^{n+1}_t \cup \{\pi_{\theta} \colon s_{\theta} \rightarrow s \, | \, \theta \colon \prod_{i=1}^m s_i \rightarrow s \in T^{n+1}\},$$
\item[]
\begin{align*}
E^{n+1} = E^{'n+1} &\cup \{\eta_{\theta}(\alpha_{\theta}(x))= \varepsilon_{\theta}(\alpha_{\theta}(x)) \, | \, \theta \in T^{n+1}\}\\
&\cup \{\pi_{\theta}(\alpha_{\theta}(x))= x \, | \, \theta \in T^{n+1}\}\\
&\cup \{\alpha_{\theta}(\pi_{\theta}(x))= x \, | \, \theta \in T^{n+1}\}\\
&\cup \{\alpha_{\theta}(\theta(x_1,\dots,x_m))= \mu_{\theta}(x_1,\dots,x_m) \, | \, \theta \colon \prod_{i=1}^m s_i \rightarrow s \in T^{n+1}\},
\end{align*}
\item[]
$$\left\{\begin{array}{ll} \Def^{n+1}(\sigma) = \Def^{'n+1}(\sigma) \text{ if } \sigma \in \Sigma^{'n+1} \setminus \Sigma^{'n+1}_t\\
\Def^{n+1}(\pi_{\theta}) = \{\eta_{\theta}(x)=\varepsilon_{\theta}(x)\} \text{ for } \theta \in T^{n+1}. \end{array}\right.$$
\end{itemize}
$$\cd{\prod_{i=1}^m s_i \ar@<2pt>[r]^-{\mu_{\theta}} \ar@{.>}[d]_-{\theta} & s_{\theta} \ar@<2pt>@{}[ld]_(.15){}="A" \ar@<1pt>@{.>}@/^/ "A";[ld]^(.4){\pi_{\theta}}
\ar@<2pt>[r]^-{\eta_{\theta}} \ar@<-2pt>[r]_-{\varepsilon_{\theta}} & s'_{\theta} \\ s \ar[ru] \ar@{}@<5pt>[ru]|-{\alpha_{\theta}}&&}$$
We have constructed $\Delta^{n+1} \subseteq \Gamma^{n+1}$ and this completes the recursive definition of
$$\Gamma^0 \subseteq \Delta^1 \subseteq \Gamma^1 \subseteq \cdots$$
Let $\Gamma_{\Mal}$ be the union of these finitary essentially algebraic theories.
By that we obviously mean $S_{\Mal}= \bigcup_{n\geqslant 0} S^n$, $\Sigma_{\Mal}= \bigcup_{n\geqslant 0} \Sigma^n$, $E_{\Mal}= \bigcup_{n\geqslant 0} E^n$,
$\Sigma_{t,\Mal}= \bigcup_{n\geqslant 0} \Sigma_t^n$ and $\Def_{\Mal}(\sigma)= \Def^n(\sigma)$ for all $n \geqslant 0$ and all $\sigma \in \Sigma^n \setminus \Sigma_t^n$.
Remark that, for each $\pi \colon s' \rightarrow s \in \Sigma_{\Mal} \setminus \Sigma_{t,\Mal}$,
there are three corresponding operation symbols in $\Sigma_{t,\Mal}$, these are $\alpha \colon s \rightarrow s'$ and $\eta,\varepsilon \colon s' \rightarrow s''$.

\begin{proposition} \label{mod gamma regular mal'tsev}
$\Mod(\Gamma_{\Mal})$ is a regular Mal'tsev category.
\end{proposition}

\begin{proof}
The `$\Delta$ ingredient' of the construction of $\Gamma_{\Mal}$ ensures that $\Mod(\Gamma_{\Mal})$ is a Mal'tsev category.
Indeed, the terms $\pi^s \circ \rho^s \colon s^3 \rightarrow s$ satisfy the conditions of Theorem~\ref{mod gamma maltsev}.

On the other hand, the `$\Gamma$ part' of $\Gamma_{\Mal}$ makes $\Mod(\Gamma_{\Mal})$ a regular category since
each finitary term $\theta$ of $\Sigma_{\Mal}$ is in $T^{n+1}$ for some $n \geqslant 0$, which makes the conditions of Proposition~\ref{mod gamma regular} hold.
\end{proof}

\section{The embedding theorem} \label{section the embedding theorem}

The aim of this section is to prove that, for each small regular Mal'tsev category $\CC$, there exists a faithful embedding
$\phi \colon \CC \hookrightarrow \Mod(\Gamma_{\Mal})^{\Sub(1)}$ which preserves and reflects finite limits, isomorphisms and regular epimorphisms.
In order to do this, we still need to recall/prove some other propositions about the embedding $\CC \hookrightarrow \Lex(\CC,\Set)^{\op}$.

\subsection{The embedding $\CC \hookrightarrow \Lex(\CC,\Set)^{\op}$} \label{subsection the embedding}

Let us now turn our attention to the Yoneda embedding $i \colon \CC \hookrightarrow \Lex(\CC,\Set)^{\op}=\widetilde{\CC}$ for a small category $\CC$ with finite limits.
Due to this embedding, we will treat $\CC$ as a full subcategory of $\widetilde{\CC}$. Firstly, let us recall the following theorems.

\begin{theorem} \cite{GU} \label{c tilde lemma} 
Let $\CC$ be a small finitely complete category. The following statements hold.
\begin{enumerate}
\item $\widetilde{\CC}$ is complete and cocomplete.
\item \label{commutativity of limits} In $\widetilde{\CC}$, cofiltered limits commute with limits and finite colimits.
\item \label{i preserves} The embedding $i \colon \CC \hookrightarrow \widetilde{\CC}$ preserves all colimits and finite limits.
\item \label{cofiltered limit} For all $A \in \widetilde{\CC}$, $(A, (c)_{(C,c)\in (A \downarrow i)} )$ is the cofiltered limit of the functor
\begin{alignat*}{2}
(A \downarrow i) & \longrightarrow && \, \widetilde{\CC}\\
c \colon A \rightarrow i(C) & \longmapsto && \, i(C).
\end{alignat*}
\item \label{i free cofiltered limit completion} $i \colon \CC \hookrightarrow \widetilde{\CC}$ is the free cofiltered limit completion of $\CC$.
\end{enumerate}
\end{theorem}

For precise definitions of cofiltered limits and their commutativity with limits and finite colimits,
we refer the reader to Sections~2.12 and~2.13 of the first volume of~\cite{borceux}.

We recall that $P \in \widetilde{\CC}$ is \textit{regular $\CC$-projective} (abbreviated here by \textit{$\CC$-projective}) if, for any diagram
$$\cd{ & P \ar[d]^-{g} \\ C \ar@{->>}[r]_-{f} & C'}$$
where $C,C' \in \CC$ and $f$ is a regular epimorphism, there exists a morphism $h \colon P \rightarrow C$ such that $fh=g$.
By the Yoneda lemma, if we consider $P$ as a finite limit preserving functor $\CC \rightarrow \Set$, morphisms $P \rightarrow C'$ in $\widetilde{\CC}$ are in
1-1 correspondence with elements of $P(C')$. Thus, $P \in \widetilde{\CC}$ is $\CC$-projective if and only if $P \colon \CC \rightarrow \Set$ preserves regular epimorphisms.

\begin{theorem} (Theorems~2.2 and~2.7 in~\cite{barr}) \label{covering}
Let $\CC$ be a small regular category. Then $\widetilde{\CC}$ is regular and each object $X \in \widetilde{\CC}$ admits a $\CC$-projective cover, i.e., a regular epimorphism
$e_X \colon \widehat{X} \twoheadrightarrow X$ where $\widehat{X}$ is $\CC$-projective.
\end{theorem}

We now prove that the regular Mal'tsev property is also `preserved' by the embedding $i \colon \CC \hookrightarrow \widetilde{\CC}$.

\begin{proposition} \label{regular maltsev preserved}
Let $\CC$ be a small regular Mal'tsev category. Then $\widetilde{\CC}$ is also a regular Mal'tsev category.
\end{proposition}

\begin{proof}
By Theorem~\ref{covering}, we already know that $\widetilde{\CC}$ is regular.
We are going to prove that $\widetilde{\CC}$ is a Mal'tsev category using Theorem~\ref{characterisations mal'tsev for regular}.\ref{mal'tsev reflexive graphs}.
So, let 
$$\cd{G \ar@<2pt>[r]^-{d} \ar@<-2pt>[r]_-{c} & X \ar@<3pt>@/^0.9pc/[l]^-{s}}$$
be a reflexive graph in $\widetilde{\CC}$.
By Lemma~5.1 in~\cite{makkai}, it is a cofiltered limit of reflexive graphs in $\CC$.
$$\cd{G \ar@<2pt>[rr]^-{d} \ar@<-2pt>[rr]_-{c} \ar[dd]_-{\lambda^1_i} && X \ar@<3pt>@/^12pt/[ll]^-{s} \ar[dd]^-{\lambda^0_i} \\ \\
G_i \ar@<2pt>[rr]^-{d_i} \ar@<-2pt>[rr]_-{c_i} && X_i \ar@<3pt>@/^12pt/[ll]^-{s_i}}$$ 
Since limits commute with limits, the pullback of $(d,c)$ along $(c,d)$ is the limit of the corresponding pullbacks arising from the reflexive graphs in $\CC$.
$$\cd{P \ar[ddd]_-{\lambda^2_i} \ar[rr] \ar[rd]_-{p} && G \ar[ddd]^(.6){\lambda^1_i}|!{[dl];[dr]}\hole \ar[rd]^-{(d,c)} & \\ & G \ar[ddd]_(.4){\lambda^1_i} \ar[rr]_(.25){(c,d)} &&
X \times X \ar[ddd]^-{\lambda^0_i \times \lambda^0_i} \\ \\
P_i \ar[rr]|!{[uur];[dr]}\hole \ar@{->>}[rd]_-{p_i} && G_i \ar[rd]^-{(d_i,c_i)} & \\ & G_i \ar[rr]_-{(c_i,d_i)} && X_i \times X_i}$$
Similarly, the kernel pair of $p$ is the cofiltered limit of the kernel pairs of the $p_i$'s.
$$\cd{R \ar[dd]_-{\lambda^3_i} \ar@<2pt>[r]^-{r} \ar@<-2pt>[r]_-{s} & P \ar[dd]^-{\lambda^2_i} \ar[r]^-{p} & G \ar[dd]^-{\lambda^1_i} \\ \\
R_i \ar@<2pt>[r]^-{r_i} \ar@<-2pt>[r]_-{s_i} & P_i \ar@{->>}[r]_-{p_i} & G_i}$$
Since $\CC$ is a Mal'tsev category, the $p_i$'s are regular epimorphisms, and so coequalisers of $r_i$ and $s_i$.
By Theorem~\ref{c tilde lemma}, cofiltered limits commute with coequalisers in $\widetilde{\CC}$.
Thus, $p$, which is the limit of the coequalisers of the $r_i$'s and $s_i$'s, is the coequaliser of their limits $r$ and $s$. Therefore $p$ is a regular epimorphism and
$\widetilde{\CC}$ is a regular Mal'tsev category.
\end{proof}

Note that this preservation property of the embedding $\CC \hookrightarrow \widetilde{\CC}$ can be generalised to a wide range of properties, see~\cite{jacqmin,JJ}.

\subsection{Proof of the embedding theorem}

We are now able to prove our main theorem.

\begin{theorem} \label{embedding theorem}
Let $\CC$ be a small regular Mal'tsev category and $\Sub(1)$ the set of subobjects of its terminal object 1. Then, there exists a faithful embedding
$\phi \colon \CC \hookrightarrow \Mod(\Gamma_{\Mal})^{\Sub(1)}$ which preserves and reflects finite limits, isomorphisms and regular epimorphisms.
Moreover, for each morphism $f \colon C \rightarrow C'$ in $\CC$, each $I \in \Sub(1)$ and each $s \in S_{\Mal}$,
$$(\Img \phi(f)_I)_s=\{(\phi(f)_I)_s(x) \, | \, x \in (\phi(C)_I)_s\}.$$
\end{theorem}

\begin{proof}
We know that $\widetilde{\CC}$ is a regular Mal'tsev category (Proposition~\ref{regular maltsev preserved}).
In what follows, we will denote by $\widehat{X}$ the $\CC$-projective covering of $X \in \widetilde{\CC}$ given by Theorem~\ref{covering}.
If $C \in \CC$ and $P \in \Sub(1)$, we are going to construct $\phi(C)_P \in \Mod(\Gamma_{\Mal})$.
More precisely, we are going to construct a $\Gamma_{\Mal}$-model $\phi(C)_P$ satisfying the following conditions:
\begin{enumerate}
\item \label{phi sort} For each $s \in S_{\Mal}$, $(\phi(C)_P)_s= \widetilde{\CC}(P_s,C)$ for some $\CC$-projective object $P_s \in \widetilde{\CC}$.
\item \label{phi alpha} For each $\pi \colon s' \rightarrow s \in \Sigma_{\Mal} \setminus \Sigma_{t,\Mal}$ and its corresponding $\alpha \colon s \rightarrow s'$, there is a given
regular epimorphism $l_{\alpha} \colon P_{s'} \twoheadrightarrow P_s$ in $\widetilde{\CC}$ such that
\begin{align*}
\alpha \colon \widetilde{\CC}(P_s,C) &\rightarrow \widetilde{\CC}(P_{s'},C)\\
f &\mapsto fl_{\alpha} \qquad \qquad \qquad \qquad \qquad \qquad
\end{align*}
and
\begin{align*}
\pi \colon \widetilde{\CC}(P_{s'},C) &\rightarrow \widetilde{\CC}(P_s,C)\\
g &\mapsto \text{the unique }f \text{ such that }fl_{\alpha}=g
\end{align*}
where $\pi$ is defined if and only if such an $f$ exists. For the corresponding $\eta,\varepsilon \colon s' \rightarrow s''$, we consider the kernel pair $(v,w)$ of $l_{\alpha}$.
$$\cd{\widehat{R} \ar@{->>}[r]^-{e_{R}} & R \ar@<2pt>[r]^-{v} \ar@<-2pt>[r]_-{w} & P_{s'} \ar@{->>}[r]^-{l_{\alpha}} & P_s}$$
We require then $P_{s''}=\widehat{R}$,
\begin{align*}
\eta \colon \widetilde{\CC}(P_{s'},C) &\rightarrow \widetilde{\CC}(P_{s''},C)\\
g &\mapsto gve_{R}
\end{align*}
and
\begin{align*}
\varepsilon \colon \widetilde{\CC}(P_{s'},C) &\rightarrow \widetilde{\CC}(P_{s''},C)\\
g &\mapsto gwe_{R}.
\end{align*}
\item \label{phi p} For each sort $s \in S_{\Mal}$, we consider the universal approximate Mal'tsev co-operation $(p^{P_s},a^{P_s})$ on $P_s$
$$\cd{\widehat{M(P_s)} \ar@{->>}[r]^-{e_{M(P_s)}}
& M(P_s) \pb \ar[r]^-{p^{P_s}} \ar@{->>}[d]_-{a^{P_s}} &
3P_s \ar[d]^-{\left( \begin{smallmatrix} \iota_1 & \iota_1 \\ \iota_2 & \iota_1 \\ \iota_2 & \iota_2 \end{smallmatrix}\right)} \\
& P_s \ar[r]_-{(\iota_1,\iota_2)} & (2P_s)^2}$$
where $a^{P_s}$ is a regular epimorphism by Theorem~\ref{approximate co-operations}.
We require then $P_{(s,0)}=\widehat{M(P_s)}$ and
\begin{align*}
\rho^s \colon \widetilde{\CC}(P_{s},C)^3 &\rightarrow \widetilde{\CC}(P_{(s,0)},C)\\
(f,g,h) &\mapsto \left(\begin{smallmatrix} f \\ g \\ h \end{smallmatrix} \right) p^{P_s}e_{M(P_s)}.
\end{align*}
\item \label{phi mu} For each finitary term $\theta \colon \prod_{i=1}^m s_i \rightarrow s$ of $\Sigma_{\Mal}$,
there is a given morphism $l_{\mu_{\theta}} \colon P_{s_{\theta}} \rightarrow P_{s_1}+ \cdots + P_{s_m}$ such that
\begin{align*}
\mu_{\theta} \colon \widetilde{\CC}(P_{s_1},C) \times \cdots \times \widetilde{\CC}(P_{s_m},C) &\rightarrow \widetilde{\CC}(P_{s_{\theta}},C)\\
(f_1,\dots,f_m) &\mapsto \left(\begin{smallmatrix}
f_1\\ \vphantom{\int\limits^x}\smash{\vdots} \\
f_m
\end{smallmatrix} \right) l_{\mu_{\theta}}.
\end{align*}
\end{enumerate}
Since $\Gamma_{\Mal}$ is the union of the series
$$\Gamma^0 \subseteq \Delta^1 \subseteq \Gamma^1 \subseteq \cdots$$
of essentially algebraic theories, to construct a $\Gamma_{\Mal}$-model $\phi(C)_P$, it is enough to construct recursively a $\Gamma^n$-model for each $n \geqslant 0$
such that they agree on the common sorts and operations. Firstly, to define a $\Gamma^0$-model, we set $P_{\star}$ to be the coproduct of the $\widehat{C'}$'s for all
$C' \in \CC$ such that the image of the unique morphism $C' \rightarrow 1$ is $P \in \Sub(1)$.
This object $P_{\star}$ is $\CC$-projective since it is the coproduct of $\CC$-projective objects.

Now, we suppose we have defined a $\Gamma^n$-model  satisfying the above conditions. We are going to extend it to a $\Gamma^{n+1}$-model with the
same properties. Firstly, we extend it to a $\Delta^{n+1}$-model. Let $s \in S^n \setminus S^{n-1}$.
Condition~\ref{phi p} above imposes the constructions of $P_{(s,0)}$ and $\rho^s$.
Moreover, condition~\ref{phi alpha} with $l_{\alpha^s}=a^{P_{s}} e_{M(P_s)}$ from condition~\ref{phi p} defines $\alpha^{s}$, $\pi^{s}$, $P_{(s,1)}$, $\eta^{s}$
and $\varepsilon^{s}$. It follows then from the definitions that this indeed gives a $\Delta^{n+1}$-model which satisfies conditions~\ref{phi sort}--\ref{phi mu}.

It remains to extend it to a $\Gamma^{n+1}$-model. In order to simplify the proof, we are going to construct $P_{s_{\theta}}$, $l_{\mu_{\theta}}$ and $l_{\alpha_{\theta}}$
for each finitary term $\theta \colon \prod_{i=1}^m s_i \rightarrow s$ of $\Sigma^{'n+1}$ such that it matches the previous construction if $\theta$ is actually a
term of $\Sigma^{'n}$. Then, condition~\ref{phi alpha} will define $\alpha_{\theta}$, $\pi_{\theta}$, $P_{s'_{\theta}}$, $\eta_{\theta}$ and $\varepsilon_{\theta}$,
while condition~\ref{phi mu} imposes the definition of $\mu_{\theta}$. We are going to do it recursively in such a way that the equality
$$\alpha_{\theta}(\theta(f_1,\dots,f_m))=\mu_{\theta}(f_1,\dots,f_m)$$
holds for any cospan $(f_i \colon P_{s_i} \rightarrow C)_{i \in \{1,\dots,m\}}$ such that $\theta(f_1,\dots,f_m)$ is defined.

Firstly, let $\theta=p_j \colon \prod_{i=1}^m s_i \rightarrow s_j$ be a projection ($1 \leqslant j \leqslant m$). In this case, we define $P_{s_{\theta}}=P_{s_j}$,
$l_{\mu_{\theta}}= \iota_j \colon P_{s_j} \rightarrow P_{s_1} + \cdots + P_{s_m}$ the coproduct injection and $l_{\alpha_{\theta}}=1_{P_{s_j}}$.
Obviously, one has
$$\alpha_{\theta}(\theta(f_1,\dots,f_m))=f_j=
\left(\begin{smallmatrix} f_1\\ \vphantom{\int\limits^x}\smash{\vdots}\\ f_m \end{smallmatrix} \right) \iota_j=\mu_{\theta}(f_1,\dots,f_m)$$
for any cospan $(f_i \colon P_{s_i} \rightarrow C)_{i \in \{1,\dots,m\}}$.

Secondly, let $\theta \colon \prod_{i=1}^m s_i \rightarrow s'$ be a finitary term of $\Sigma^{'n+1}$ for which $l_{\mu_{\theta}}$ and $l_{\alpha_{\theta}}$ have been constructed.
If $\pi \colon s' \rightarrow s \in \Sigma^{'n+1} \setminus \Sigma_t^{'n+1}$ has corresponding $\alpha \colon s \rightarrow s'$,
we define $P_{s_{\pi(\theta)}}=P_{s_{\theta}}$, $l_{\alpha_{\pi(\theta)}}=l_{\alpha} l_{\alpha_{\theta}}$ and $l_{\mu_{\pi(\theta)}}=l_{\mu_{\theta}}$.
$$\cd{P_{s_{\pi(\theta)}}=P_{s_{\theta}} \ar[rr]^-{l_{\mu_{\pi(\theta)}}=l_{\mu_{\theta}}} \ar@{->>}[d]_-{l_{\alpha_{\theta}}} \ar@{->>}[rd]^-{l_{\alpha_{\pi(\theta)}}}
&& P_{s_1} + \cdots + P_{s_m}\\
P_{s'} \ar@{->>}[r]_-{l_{\alpha}} & P_s &}$$
If the cospan $(f_i \colon P_{s_i} \rightarrow C)_{i \in \{1,\dots,m\}}$ is such that $\theta(f_1,\dots,f_m) \colon P_{s'} \rightarrow C$ is defined,
we know from the previous step in the recursion that
$$\theta(f_1,\dots,f_m) l_{\alpha_{\theta}} = \alpha_{\theta}(\theta(f_1,\dots,f_m)) = \mu_{\theta}(f_1,\dots,f_m).$$
If moreover $\pi(\theta(f_1,\dots,f_m)) \colon P_s \rightarrow C$ is defined, we have $$\pi(\theta(f_1,\dots,f_m)) l_{\alpha} = \theta(f_1,\dots,f_m).$$
In this case,
\begin{align*}
\alpha_{\pi(\theta)}(\pi(\theta(f_1,\dots,f_m))) &= \pi(\theta(f_1,\dots,f_m)) l_{\alpha_{\pi(\theta)}}\\
&= \pi(\theta(f_1,\dots,f_m)) l_{\alpha} l_{\alpha_{\theta}}\\
&= \theta(f_1,\dots,f_m) l_{\alpha_{\theta}}\\
&= \mu_{\theta}(f_1,\dots,f_m)\\
&= \mu_{\pi(\theta)}(f_1,\dots,f_m).
\end{align*}

Eventually, let us suppose that $\sigma \colon \prod_{i=1}^r s'_i \rightarrow s \in \Sigma_t^{'n+1}$ is an operation symbol and, for each $1 \leqslant j \leqslant r$,
$\theta_j \colon \prod_{i=1}^m s_i \rightarrow s'_j$ is a finitary term of $\Sigma^{'n+1}$ for which $l_{\mu_{\theta_j}}$ and $l_{\alpha_{\theta_j}}$ have been defined.
We already have a corresponding morphism $l_{\sigma} \colon P_s \rightarrow P_{s'_1} + \cdots + P_{s'_r}$ such that
\begin{align*}
\sigma \colon \widetilde{\CC}(P_{s'_1},C) \times \cdots \times \widetilde{\CC}(P_{s'_r},C) &\rightarrow \widetilde{\CC}(P_s,C)\\
(f_1,\dots,f_r) &\mapsto \left(\begin{smallmatrix}
f_1\\ \vphantom{\int\limits^x}\smash{\vdots} \\ f_r
\end{smallmatrix} \right) l_{\sigma}.
\end{align*}
Let us consider the following diagram where the square is a pullback.
$$\cd{P_{s_{\theta}}=\widehat{U} \ar@{->>}[r]^-{e_U} & U \pb \ar@{->>}[d]_-{u_2} \ar[r]^-{u_1} & P_{s_{\theta_1}} + \cdots + P_{s_{\theta_r}}
\ar@{->>}[d]^-{l_{\alpha_{\theta_1}} + \cdots + l_{\alpha_{\theta_r}}}
\ar[rr]^-{\left(\begin{smallmatrix} l_{\mu_{\theta_1}} \\ \vphantom{\int\limits^x}\smash{\vdots}\\ l_{\mu_{\theta_r}} \end{smallmatrix} \right)} && P_{s_1} + \cdots + P_{s_m}\\
& P_s \ar[r]_-{l_{\sigma}} & P_{s'_1} + \cdots + P_{s'_r} &&}$$
Denoting the term $\sigma(\theta_1,\dots,\theta_r) \colon \prod_{i=1}^m s_i \rightarrow s$ by $\theta$, we define $P_{s_{\theta}}=\widehat{U}$, $l_{\alpha_{\theta}}=u_2 e_U$ and
$$l_{\mu_{\theta}}=\left(\begin{smallmatrix} l_{\mu_{\theta_1}} \\ \vphantom{\int\limits^x}\smash{\vdots} \\ l_{\mu_{\theta_r}} \end{smallmatrix} \right)u_1e_U.$$
Then, if the cospan $(f_i \colon P_{s_i} \rightarrow C)_{i \in \{1,\dots,m\}}$ is such that $\theta_j(f_1,\dots,f_m) \colon P_{s'_j} \rightarrow C$ is defined
for all $1 \leqslant j \leqslant r$,
{\allowdisplaybreaks
\begin{align*}
\alpha_{\theta}(\theta(f_1,\dots,f_m))&= \sigma(\theta_1(f_1,\dots,f_m),\dots,\theta_r(f_1,\dots,f_m)) l_{\alpha_{\theta}}\\*
&= \left(\begin{smallmatrix} \theta_1(f_1,\dots,f_m) \\ \vphantom{\int\limits^x}\smash{\vdots} \\ \theta_r(f_1,\dots,f_m) \end{smallmatrix} \right) l_{\sigma} u_2 e_U\\
&= \left(\begin{smallmatrix} \theta_1(f_1,\dots,f_m) \\ \vphantom{\int\limits^x}\smash{\vdots} \\ \theta_r(f_1,\dots,f_m) \end{smallmatrix} \right)
(l_{\alpha_{\theta_1}} + \cdots + l_{\alpha_{\theta_r}}) u_1 e_U\\
&= \left(\begin{smallmatrix} \alpha_{\theta_1}(\theta_1(f_1,\dots,f_m)) \\ \vphantom{\int\limits^x}\smash{\vdots}
\\ \alpha_{\theta_r}(\theta_r(f_1,\dots,f_m)) \end{smallmatrix} \right) u_1 e_U\\
&= \left(\begin{smallmatrix} \mu_{\theta_1}(f_1,\dots,f_m) \\ \vphantom{\int\limits^x}\smash{\vdots} \\ \mu_{\theta_r}(f_1,\dots,f_m) \end{smallmatrix} \right) u_1 e_U\\
&= \left(\begin{smallmatrix} f_1 \\ \vphantom{\int\limits^x}\smash{\vdots} \\ f_m \end{smallmatrix} \right)
\left(\begin{smallmatrix} l_{\mu_{\theta_1}} \\ \vphantom{\int\limits^x}\smash{\vdots} \\ l_{\mu_{\theta_r}} \end{smallmatrix} \right)u_1e_U\\*
&= \mu_{\theta}(f_1,\dots,f_m)
\end{align*}}%
using the previous steps in the recursion.

We have thus defined a $\Gamma^{n+1}$-model which satisfies conditions~\ref{phi sort}--\ref{phi mu}.
This concludes the recursive construction of our $\Gamma^n$-model for each $n \geqslant 0$. Considering them all together, we get a $\Gamma_{\Mal}$-model $\phi(C)_P$. 

Now, if $f \colon C \rightarrow C' \in \CC$ and $P \in \Sub(1)$, we define a morphism $\phi(f)_P \colon \phi(C)_P \rightarrow \phi(C')_P$ by
\begin{align*}
(\phi(f)_P)_s \colon \widetilde{\CC}(P_s,C) &\rightarrow \widetilde{\CC}(P_s,C')\\
g &\mapsto fg
\end{align*}
for all $s \in S_{\Mal}$. By conditions~\ref{phi alpha}--\ref{phi mu}, $\phi(f)_P$ is a $\Gamma_{\Mal}$-homomorphism. This defines the expected
functor $\phi \colon \CC \rightarrow \Mod(\Gamma_{\Mal})^{\Sub(1)}$. Let us now check that it satisfies all the required properties.

Since finite limits in $\Mod(\Gamma_{\Mal})^{\Sub(1)}$ are computed componentwise, to prove that $\phi$ preserves them, we only need to prove that
$\phi(-)_P \colon \CC \rightarrow \Mod(\Gamma_{\Mal})$ preserves finite limits for each $P \in \Sub(1)$.
Furthermore, since they are computed in each sort as in $\Set$, we only need to check that $(\phi(-)_P)_s \colon \CC \rightarrow \Set$ preserves finite limits
for all $P \in \Sub(1)$ and all $s \in S_{\Mal}$.
But, by the Yoneda lemma, $(\phi(-)_P)_s$ is isomorphic to $P_s \colon \CC \rightarrow \Set$ which preserves finite limits by definition.
Therefore, $\phi$ preserves them as well. 

Now, suppose that $f \colon C \rightarrow C' \in \CC$ is such that $(\phi(f)_P)_s$ is surjective for all $P \in \Sub(1)$ and all $s \in S_{\Mal}$.
Let $$\cd{C' \ar@{->>}[r]^-{p} & I \ar@{}[r]|(.25){}="A" \ar@{>->}"A";[r] & 1}$$ be the image factorisation of the unique morphism $C' \rightarrow 1$.
We recall that $I_{\star} = \coprod {\widehat{C''}}$ where the coproduct runs through all $C''$ such that the image of $C'' \rightarrow 1$ is $I$.
For each such $C''$, there exists a morphism $g_{C''}$ making the diagram
$$\cd{\widehat{C''} \ar@{->>}[r]^-{e_{C''}} \ar[d]_-{g_{C''}} & C'' \ar@{->>}[d] \\ C' \ar@{->>}[r]_-{p} & I \ar@{}[r]|(.25){}="A" \ar@{>->}"A";[r] & 1}$$
commutative since $\widehat{C''}$ is $\CC$-projective. We choose $g_{C'}=e_{C'}$ and consider the induced morphism $g \colon I_{\star} \twoheadrightarrow C'$
which is a regular epimorphism since $g \iota_{\widehat{C'}}=g_{C'}$ is.
But we have supposed that $f \circ - \colon \CC(I_{\star},C) \rightarrow \CC(I_{\star},C')$ is surjective.
So, there is a morphism $h \colon I_{\star} \rightarrow C$ such that $fh=g$, which implies that $f$ is also a regular epimorphism.
Moreover, since each $P_s$ is $\CC$-projective, this means that $f$ is a regular epimorphism
if and only if $(\phi(f)_P)_s$ is surjective for all $P \in \Sub(1)$ and all $s \in S_{\Mal}$.
In particular, $\phi$ preserves regular epimorphisms.

Now, let $f,f' \colon C \rightarrow C'$ be two morphisms of $\CC$ such that $(\phi(f)_P)_s=(\phi(f')_P)_s$ for all $P \in \Sub(1)$ and all $s \in S_{\Mal}$.
Let $e \colon E \rightarrowtail C$ be their equaliser.
Since $\phi$ preserves equalisers, $(\phi(e)_P)_s$ is a bijection for all $P \in \Sub(1)$ and all $s \in S_{\Mal}$. Hence, $e$ is a regular epimorphism and $f=f'$.
This shows that $\phi$ is faithful.

Let $f \colon C \rightarrow D \in \CC$ be such that $(\phi(f)_P)_s$ is injective for all $P \in \Sub(1)$ and all $s \in S_{\Mal}$. We want to prove that $f$ is a monomorphism.
So, suppose $h,k \colon C' \rightarrow C \in \CC$ are such that $fh=fk$. Let $g \colon I_{\star} \twoheadrightarrow C'$ be the regular epimorphism defined as above.
Thus, $fhg=fkg$.
Since $f \circ - \colon \widetilde{\CC}(I_{\star},C) \rightarrow \widetilde{\CC}(I_{\star},D)$
is injective, we know that $hg=kg$. Hence $h=k$ and $f$ is a monomorphism since $g$ is a regular epimorphism.
Therefore, $\phi$ reflects isomorphisms, finite limits and regular epimorphisms.

It remains to check that, for $f \colon C \rightarrow C' \in \CC$, $P \in \Sub(1)$ and $s \in S_{\Mal}$, 
$$(\Img \phi(f)_P)_s=\{(\phi(f)_P)_s(x) \, | \, x \in (\phi(C)_P)_s\}.$$
Consider $\pi \colon s' \rightarrow s \in \Sigma_{\Mal} \setminus \Sigma_{t,\Mal}$ and $x \in \widetilde{\CC}(P_{s'},C)$ such that $\pi((\phi(f)_P)_{s'}(x))$ is defined.
So, there exists $g \colon P_s \rightarrow C'$ making the square
$$\cd{P_{s'} \ar[r]^-{x} \ar@{->>}[d]_-{l_{\alpha}} & C \ar[d]^-{f} \\ P_s \ar[r]_-{g} & C'}$$
commute (with $\alpha \colon s \rightarrow s'$ corresponding to $\pi$).
Let $f=iq$ be the image factorisation of $f$. Since $l_{\alpha}$ is a regular epimorphism, there exists $g' \colon P_s \rightarrow \Img(f)$
such that $ig'=g$. Since $P_s$ is $\CC$-projective, there exists a morphism $y \colon P_s \rightarrow C$ such that $qy=g'$. Thus, $fy=g$ and
$(\phi(f)_P)_s(y)=g=\pi((\phi(f)_P)_{s'}(x))$. Therefore, in view of the description of the images in categories of $\Gamma$-models on page \pageref{image} for any
finitary essentially algebraic theory $\Gamma$, this concludes the proof.
\end{proof}

\section{Applications} \label{section applications}

Analogously to the metatheorems of~\cite{BB},
our embedding theorem gives a way to prove some statements in regular Mal'tsev categories in an `essentially algebraic way' as follows.

Consider a statement $P$ of the form $\psi \Rightarrow \omega$ where $\psi$ and $\omega$ are conjunctions of properties which can be expressed as
\begin{enumerate}
\item some finite diagram is commutative,
\item some finite diagram is a limit diagram,
\item some morphism is a monomorphism,
\item some morphism is a regular epimorphism,
\item some morphism is an isomorphism,
\item some morphism factors through a given monomorphism.
\end{enumerate}
Then, this statement $P$ is valid in all regular Mal'tsev $\V$-categories (for all universes $\V$) if and only if it is valid in $\Mod(\Gamma_{\Mal})$ (for all universes).
Indeed, in view of Proposition~\ref{mod gamma regular mal'tsev}, the `only if part' is obvious. Conversely, if $\CC$ is a regular Mal'tsev category,
we can consider it is small up to a change of universe. Then, by Theorem~\ref{embedding theorem}, it suffices to prove $P$ in $\Mod(\Gamma_{\Mal})^{\Sub(1)}$.
Since every part of the statement $P$ is `componentwise', it is enough to prove it in $\Mod(\Gamma_{\Mal})$.

At a first glance, one could think this technique will be hard to use in practice, in view of the difficult definition of $\Mod(\Gamma_{\Mal})$.
However, due to the additional property in our Theorem~\ref{embedding theorem}, we can suppose that the homomorphisms $f \colon A \rightarrow B$ considered in the statement $P$
have an easy description of their images, i.e., $$(\Img f)_s = \{f_s(a) \, | \, a \in A_s\}$$ for all $s \in S_{\Mal}$.
In particular, if $f$ is a regular epimorphism, $f_s$ will be a surjective function for all $s \in S_{\Mal}$.
Therefore, in practice, it seems we will never have to use the operations $\alpha_{\theta}$, $\mu_{\theta}$, $\eta_{\theta}$, $\varepsilon_{\theta}$
and $\pi_{\theta}$. They were built only to make $\Mod(\Gamma_{\Mal})$ a regular category.

We illustrate now how to use the embedding theorem to prove a result in a regular Mal'tsev category using an (essentially) algebraic argument.
We refer the reader to~\cite{BB} for a definition of the category $\Pt(\CC)$ of points of $\CC$.

\begin{lemma} (Proposition~4.1 in~\cite{bourn})
Let $\CC$ be a regular Mal'tsev category. Consider a commutative square in the category $\Pt(\CC)$ of points of $\CC$. This yields a cube where the vertical morphisms are split
epimorphisms with a given section. Suppose that the left and right faces of this cube are pullbacks and $p$ and $q$ are regular epimorphisms.
$$\xymatrix@C=30pt@R=20pt{
    X \times_Y Z \ar@<-2pt>@{->>}[dd] \ar[dr] \ar[rr]^-{t} & & U \times_V W \ar@<-2pt>@{->>}[dd]|!{[dl];[dr]}\hole \ar[dr] \\
    & Z \ar@<-2pt>@{->>}[dd]_(.65){g} \ar@{->>}[rr]^(.25){q} & & W \ar@<-2pt>@{->>}[dd]_-{k} \\
    X \ar@<-2pt>[uu] \ar[dr]_-{f} \ar@{->>}[rr]^(.7){p}|!{[ru];[dr]}\hole|!{[dr];[ur]}\hole & & U \ar@<-2pt>[uu]|!{[ul];[ur]}\hole \ar[dr]_(.4){h} \\
    & Y \ar@<-2pt>[uu]_(.35){g'} \ar@{->>}[rr]_-{r} & & V \ar@<-2pt>[uu]_-{k'} }$$
Then, the comparison map $t$ is also a regular epimorphism.
\end{lemma}

\begin{proof}
It is enough to prove this lemma in $\Mod(\Gamma_{\Mal})$ supposing that $p$ and $q$ are surjective in each sort.
So, let $s \in S_{\Mal}$, $u \in U_s$ and $w \in W_s$ be such that $h(u)=k(w)$ and let us prove $(u,w) \in \Img(t)_s$.
Since $p$ and $q$ are surjective, we can find $x \in X_s$ and $z \in Z_s$ such that $p(x)=u$ and $q(z)=w$.
Let $z'=\rho^s(g'f(x),g'g(z),z) \in Z_{(s,0)}$. Since $g(z')=\rho^s(f(x),g(z),g(z))=\alpha^s(f(x))=f(\alpha^s(x)),$
we can consider $(\alpha^s(x),z') \in (X \times_Y Z)_{(s,0)}$. Moreover, since
\begin{align*}
q(z') &= \rho^s(qg'f(x),qg'g(z),q(z))\\
&= \rho^s(k'hp(x),k'kq(z),q(z))\\
&= \rho^s(k'h(u),k'k(w),w)\\
&= \rho^s(k'k(w),k'k(w),w)\\
&= \alpha^s(w),
\end{align*}
we know that $t(\alpha^s(x),z')=(p(\alpha^s(x)),q(z'))=(\alpha^s(u),\alpha^s(w))=\alpha^s(u,w)$.
Therefore, we have $(u,w)=\pi^s(\alpha^s(u,w))=\pi^s(t(\alpha^s(x),z')) \in \Img(t)_s$.
\end{proof}




\begin{thebibliography}{99}
\addcontentsline{toc}{section}{References}

\bibitem{AR}
\textsc{J. Ad\'amek and J. Rosick\'y}, Locally presentable and accessible categories, \textit{London Math. Soc.} \textbf{189} (1994).

\bibitem{barr2} 
\textsc{M. Barr}, Exact categories, \textit{Springer Lect. Notes Math.} \textbf{236} (1971), 1--120.

\bibitem{barr}
\textsc{M. Barr}, Representation of categories, \textit{J. Pure Appl. Algebra} \textbf{41} (1986), 113--137.

\bibitem{borceux}
\textsc{F. Borceux}, Handbook of categorical algebra 1 and 2, \textit{Cambridge University Press} (1994).

\bibitem{BB}
\textsc{F. Borceux and D. Bourn}, Mal'cev, protomodular, homological and semi-abelian categories, \textit{Mathematics and Its Applications} \textbf{566} (2004).

\bibitem{bourn}
\textsc{D. Bourn}, The denormalized $3 \times 3$ lemma, \textit{J. Pure Appl. Algebra} \textbf{177} (2003), 113--129.

\bibitem{BJ}
\textsc{D. Bourn and Z. Janelidze}, Approximate Mal'tsev operations, \textit{Theory and Appl. of Categ.} \textbf{21} No. 8 (2008), 152--171.

\bibitem{CLP}
\textsc{A. Carboni, J. Lambek and M.C. Pedicchio}, Diagram chasing in Mal'cev categories, \textit{J. Pure Appl. Algebra} \textbf{69} (1990), 271--284.

\bibitem{CPP}
\textsc{A. Carboni, M.C. Pedicchio and N. Pirovano}, Internal graphs and internal groupoids in Mal'tsev categories, \textit{Canadian Math. Soc. Conf. Proc.}
\textbf{13} (1992), 97--109.

\bibitem{CPR}
\textsc{A. Carboni, M.C. Pedicchio} and \textsc{J. Rosick\'y}, Syntactic characterizations of various classes of locally presentable categories,
\textit{J. Pure Appl. Algebra} \textbf{161} (2001), 65--90.

\bibitem{freyd}
\textsc{P. Freyd}, Abelian categories, \textit{Harper and Row} (1964).

\bibitem{GU}
\textsc{P. Gabriel and F. Ulmer}, Lokal pr\"asentierbare kategorien, \textit{Springer-Verlag} (1971).

\bibitem{GP}
\textsc{M. Gran} and \textsc{M.C. Pedicchio}, $n$-Permutable locally finitely presentable categories, \textit{Theory and Appl. of Categ.} \textbf{8} No. 1 (2001), 1--15.

\bibitem{grothendieck 1}
\textsc{A. Grothendieck}, Sur quelques points d'alg\`ebre homologique, \textit{Toh\^oku Math. J.} \textbf{2} (1957), 199--221.

\bibitem{jacqmin}
\textsc{P.-A. Jacqmin}, Embedding theorems in non-abelian categorical algebra, \textit{PhD thesis, UCL} (2016).

\bibitem{JJ}
\textsc{P.-A. Jacqmin and Z. Janelidze}, Unconditional exactness properties, \textit{in preparation}.

\bibitem{lubkin}
\textsc{S. Lubkin}, Imbedding of abelian categories, \textit{Trans. Amer. Math. Soc.} \textbf{97} (1960), 410--417.

\bibitem{makkai}
\textsc{M. Makkai}, Strong conceptual completeness for first order logic, \textit{Ann. Pure Appl. Logic} \textbf{40} (1988), 167--215.

\bibitem{maltsev}
\textsc{A. I. Mal'tsev}, On the general theory of algebraic systems, \textit{Mat. Sbornik}, N.S. \textbf{35 (77)} (1954), 3--20 (in Russian);
English translation: \textit{American Mathematical Society Translations} (2) \textbf{27} (1963), 125--142.

\bibitem{mitchell}
\textsc{B. Mitchell}, The full imbedding theorem, \textit{Amer. J. Math.} \textbf{86} (1964), 619--637.

\end{thebibliography}
\end{document}